\documentclass[draft]{amsart}

\usepackage{latexsym}
\usepackage{amsmath}
\usepackage{amsfonts}
\usepackage{amsthm}
\usepackage{amssymb}
\usepackage[final]{graphicx}
\usepackage{verbatim}

\graphicspath{ {C:/Users/tazeen/Desktop/Nayab/} }
\DeclareGraphicsExtensions{.eps,.png,.jpg}

\usepackage{srcltx}

\def\R{{\mathbb R}}

\numberwithin{equation}{section}
\newcounter{thm}[section]
\numberwithin{thm}{section}

\theoremstyle{plain}
\newtheorem{theorem}[thm]{Theorem}

\newtheorem{proposition}[thm]{Proposition}

\theoremstyle{definition}
\newtheorem*{cons}{Construction}

 \title{Box-Counting Dimension Using Triangles}
\author{T. Athar}
\address{Department of Mathematics, 
COMSATS Institute of Information Technology, 
Islamabad, Pakistan}
\email{tazeen@comsats.edu.pk}

\author{N. Khalid}
\address{School of Mathematics and Statistics, 
University of St. Andrews, 
Scotland, UK}
\email{nk58@st-andrews.ac.uk}

\author{S. Ul Islam}
\address{Department of Mathematics, 
COMSATS Institute of Information Technology, 
Islamabad, Pakistan}
\email{islam$\_$shams@comsats.edu.pk}



\begin{document}

\maketitle 

\begin{abstract}
An alternate definition of the box-counting dimension is proposed, to provide a better approximation for fractals involving rotation such as the 'Bradley Spiral' structure. A curve fitting comparison of this definition with the box-counting dimension is also presented.
\end{abstract}

\section{Introduction and Preliminaries}

Fractals are an interesting mathematical structure and are gaining more attention day by day. A \textit{fractal} is a set that possesses a non-ending repeating pattern. A set may be formally defined as a fractal if its dimension is fractional. In addition to this, fractals hold many interesting properties. Most fractals are self-similar, and can be defined either iteratively or recursively. The fractal dimension is mostly calculated using either the Hausdorff dimension or box-counting dimension.  The former one is more stable, but the latter is easier to calculate. These and  other definitions of dimension can be seen in detail in \cite{falc}.
 
For a bounded subset $F$ of $\R^n$, let $N_\delta(F)$ denote the smallest number of boxes of side length $\delta$ covering $F$ . Then, by \cite{falc} the lower box-counting dimension of $F$ is defined as
\[ \underline{\dim}_B(F) = \varliminf_{\delta \to 0} \frac{\log N_\delta (F)}{- \log \delta}. \]
Similarly, upper box-counting dimension of $F$ is given as
\[ \overline{\dim}_B(F) = \varlimsup_{\delta \to 0} \frac{\log N_\delta (F)}{- \log \delta}. \]
The box-counting dimension of $F$ denoted as $\dim_{B}(F)$ (also known as box dimension or capacity dimension) is then defined as
\begin{equation} \label{eq:01}
\underline{\dim}_B(F)=\overline{\dim}_B(F)=\dim_{B}(F) = \lim_{\delta \to 0}  \frac{\log N_\delta (F)}{- \log \delta}. 
\end{equation}
There exists structures possessing rotataion of the portions being removed for which the existing definition of the box dimension using $\delta$-mesh cubes is hard to calculate. The aim of this paper is to deal with this problem by presenting an alternate definition of the box dimension, initially proposed by \cite{nk}. This article extends the work of \cite{nk} by providing a curve fitting comparison.


In the following section we provide an alternate definition of box dimension which will be used for a particular structure in Section 3.

\section{Box-Counting Dimension in Terms of Triangles}

Recall from \eqref{eq:01} that for a bounded subset $F$ of $\R^n$,  
 $N_{\delta}(F)$ (as mentioned in \cite{falc}) can be any of the following:
\begin{itemize}
\item[(i)] closed balls of radius $\delta$ that cover $F$, \notag 
\item[(ii)] cubes of side $\delta$ that cover $F$, \notag 
\item[(iii)] $\delta$-mesh cubes that intersect $F$, \notag 
\item[(iv)] sets of diameter at most $\delta$ that cover $F$, \notag 
\item[(v)] disjoint balls of radius $\delta$ with centres in $F$. \notag
\end{itemize}
Note that in (i)-(iv), $N_{\delta}(F)$ is the smallest number and in (v) the largest number of such sets.

For two-dimensional objects, the box dimension is usually calculated using $\delta$-mesh squares. The following theorem as mentioned in \cite{nk} proposes a new variant of the box counting dimension using $\delta$-mesh triangles. For this, let $T_\delta(F)$ denote the number of right-angled triangles of side length $\delta$ having non-empty intersection with $F$.

\begin{theorem}
Let $F$ be a bounded subset of $\R^2$. The lower and upper box-counting dimensions of $F$ in terms of $T_\delta(F)$ are given by
\[ \underline{\dim}_B(F) = \varliminf_{\delta \to 0} \frac{\log T_\delta (F)}{- \log \delta} \] and 
\[ \overline{\dim}_B(F) = \varlimsup_{\delta \to 0} \frac{\log T_\delta (F)}{- \log \delta}. \]
Moreover, if limit exists , we have the box-counting dimension as 
\begin{equation} \label{eq: 01}
 \dim _B (F) = \lim _{\delta \to 0} \frac{\log T_{\delta}(F)}{- \log \delta}.
\end{equation} 
\end{theorem}

The proof is similar to the proof for $N_\delta(F)$ as for (iii) in \cite{falc}.

\begin{proof}
Consider the $\delta$-mesh squares of the form 
\[ [m \delta, (m + 1) \delta] \times [n \delta, (n + 1) \delta], \]
where $m$ and $n$ are integers. Clearly each square gives rise
 to two right angled (isoceles) triangles each of shorter side $\delta$ and longer side $\delta \sqrt{2}$. Let $N_\delta(F)$ and $T_\delta(F)$ be the number of squares and triangles, respectively, covering $F$. Then, 
 \[ N_\delta(F) \leq T_\delta (F) \leq 2 N_\delta (F). \]
Since the collection $N_\delta(F)$ gives rise to $\delta$-squares of diameter $\delta \sqrt 2$, it can be obviously seen that the collection of sets $N_{\delta \sqrt{2}}(F)$ covering $F$ gives the relation 
\[ N_{\delta \sqrt{2}}(F) \leq N_{\delta}(F) \leq T_{\delta}(F).\]
Assuming that $\delta > 0$ is sufficiently small we have $\delta \sqrt{2}<1.$ Thus,
\[ \frac{\log N_{\delta \sqrt{2}}(F)}{- \log (\delta \sqrt{2})} \leq \frac{\log T_{\delta}(F)}{- \log  \sqrt{2}- \log \delta}. \]
Taking limits as $\delta \to 0$, we obtain
\[ \underline{\dim}_B(F) \leq \varliminf_{\delta \to 0} \frac{\log T_\delta (F)}{- \log \delta} \] and 
\[ \overline{\dim}_B(F) \leq \varlimsup_{\delta \to 0} \frac{\log T_\delta (F)}{- \log \delta}. \]
Now, any set of diameter at most $\delta$ is contained in $2^2  \medspace \delta$-mesh cubes. Thus,
\[ \frac{T_\delta (F)}{2} \leq N_\delta(F) \leq 2^2  N_{\delta / \sqrt 2} (F). \]
Taking logarithms, we will get
\[ \frac{\log T_\delta(F) -\log 2}{- \log \delta - \log \sqrt 2} \leq \frac{2 \log 2 + \log N_{\delta/ \sqrt 2}(F)}{- \log \delta - \log \sqrt 2}.\]
Letting $ \delta \to 0$, we obtain the other side, i.e.
\[ \underline{\dim}_B (F) \geq \varliminf_{\delta \to 0} \frac{\log T_\delta (F)}{- \log \delta} \] and
\[ \overline{\dim}_B (F) \geq  \varlimsup_{\delta \to 0} \frac{\log T_\delta (F)}{- \log \delta}. \]
 Hence, the box-counting dimension can equivalently be defined using a collection of $\delta$-mesh triangles.
\end{proof}

\section{A Spiral As A Fractal}

In this section we consider a spiral as a fractal. The construction of the spiral is taken from \cite{brad} and is assigned the named 'Bradley Spiral'. The portions being removed from the initial set follow rotation, resulting in the spiral shape. The construction is as follows:

\begin{cons}
Consider a closed unit square $N_0$. Contruct a square $N_1$ inscribed in $N_0$ by joining the midpoints of each side of $N_0$. This produces four right-angled isoceles triangles with length of shorter side $\frac{1}{2}$. Remove the upper right triangle, say $T_1$, keeping the boundary of the longer side (i.e. the boundary intersecting with $N_1$).

Similarly to the previous step, form a square $N_2$ inscribed in $N_1$, yielding four right-angled isoceles triangles with length of shorter side $\frac{1}{2^{3/2}}$. Remove the right triangle, say $T_2$, and remove the boundary intersecting with $T_1$ (keeping the other two). In the third stage, a square $N_3$ inscribed in $N_2$ is formed by joining the midpoints of each side of $N_2$. This leaves four right-angled isoceles triangles with length of shorter side $\frac{1}{2^2}$. Remove the lower right triangle, say $T_3$, and the boundary intersecting with $T_2$. 

Continuing iteratively, at the $k$-th stage, a square $N_k$ inscribed in $N_{k-1}$ is formed yielding four right-angled isoceles triangles with length of shorter side $\frac{1}{2^{(k + 1)/2}}$. One triangle, say $T_k$, is removed along with the boundary of $T_k$ intersecting with $T_{k-1}$. 

Thus, the triangles removed form a spiral and the resulting set is named the 'Bradley Spiral', and is denoted by $S$. 
\end{cons}
Figure 1 shows first three steps and the pattern obtained at $k$-th step of the Bradley spiral. Now, we calculate the box dimension of the Bradley spiral in terms of triangles.
\begin{figure}[h]
\label{Figure-1}
\scalebox{0.35}
{\includegraphics{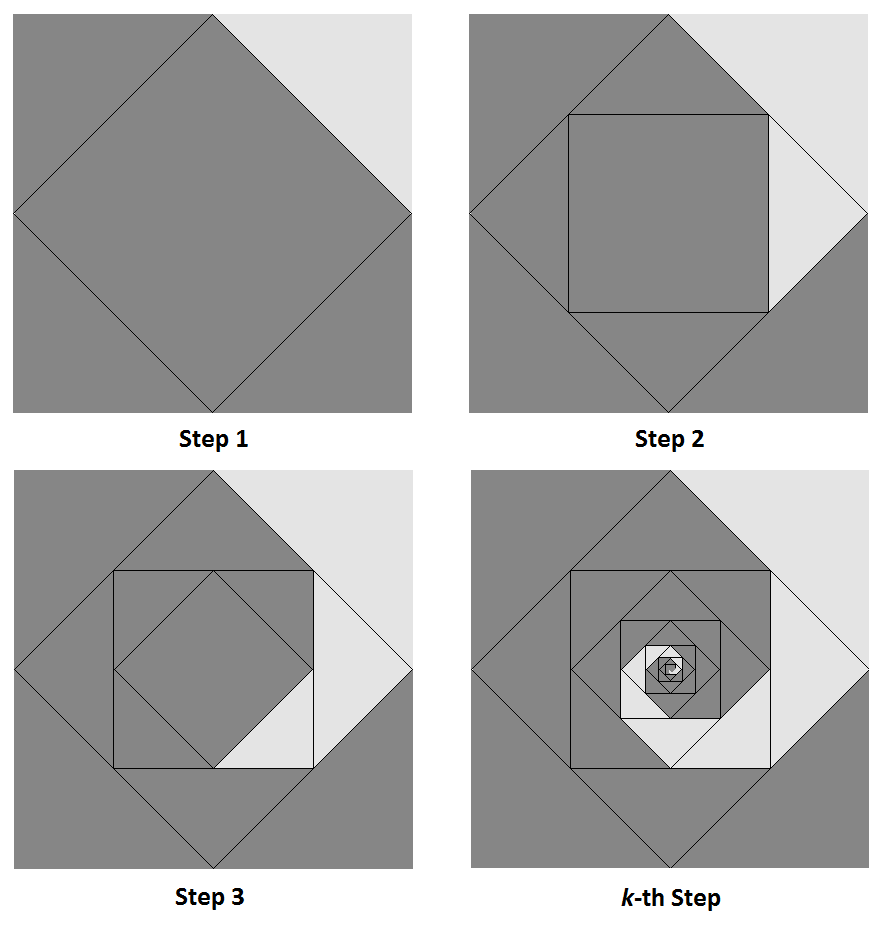}}
\caption{Bradley Spiral}
\end{figure}

\begin{proposition}
The Bradley spiral has box dimension $2$.
\end{proposition}

\begin{proof}
At the $k$-th stage, {\it S} is covered by $3 \cdot 2^k + 1$ $\delta$-mesh triangles of length $2^{- (k+1)/2}$. 
Then by \eqref{eq: 01}
\begin{align}
\dim_B(S) &= \lim_{\delta \to 0} \frac{\log T_\delta (S)}{- \log \delta} \notag \\
&= \lim_{k \to \infty} \frac{\log (3 \cdot 2^k + 1)}{- \log 2^{- (k+1)/2}} \notag
\end{align}
Observe that for large $k$, $\log (3 \cdot 2^k + 1) \approx \log (3 \cdot 2^k)$. Thus,
\begin{align}
\dim_B(S) &= \lim_{k \to \infty} \frac{\log (3 \cdot 2^k)}{\frac{(k+1)}{2} \log 2} \label{eq: 02} \\
&= \lim_{k \to \infty} \frac{2(\log 3 + k \log 2)}{(k+1) \log 2} \notag \\
&= \lim_{k \to \infty} \left[ \frac{2 k}{k + 1} + \frac{2 \log 3}{(k + 1) \log 2} \right] \notag \\
&= 2. \notag
\end{align}
\end{proof}

\section{Algorithm And Curve Fitting}

In this section, we present an algorithm for calculating the box dimension using triangle mesh of the Bradley Spiral, and also show a curve fitting comparison for the box dimension for both the triangle mesh and the square mesh.

For the algorithm, we choose equal spacing for triangle meshes $(\delta)$. The implementation is straight forward in MATLAB using the $\lim$ and $\log_{10}$ MATLAB commands. We have  also verified the results for \eqref{eq: 02} numerically and the corresponding algorithm is given below:

\begin{verbatim}
Input=k
Output=B
For i=1:k
      calculate (3.1)
    If (B >= 2) Then
        STOP
      End
  End
\end{verbatim} 

We now discuss the curve fitting for the box dimension for both the triangle mesh and the square mesh. The curve fitting here is computed from the linear least-square regression obtained for the triangle and square mesh and presented in Figure 2. We observe that the value for $\log N_\delta$ is found to increase linearly with an increase in $- \log \delta$, which is in good agreement with the linear least-square regression curve. Further, it is to be noted that, in case of the square mesh the change in $\delta$ is greater than in the triangle mesh.

\begin{figure}[h]
\label{Figure 2 (a,b)}
\scalebox{0.65}
{\includegraphics{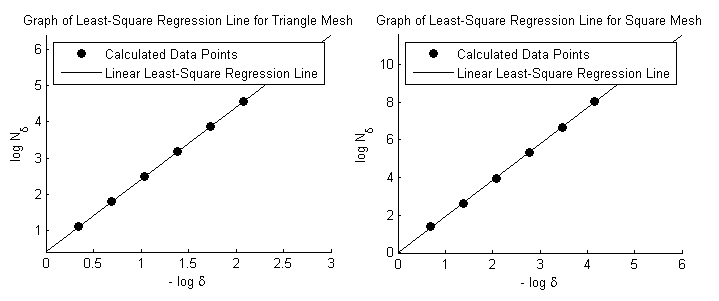}}
\caption{Curve Fitting Comparison}
\end{figure}

In conclusion, we observe that, for structures involving rotation, this alternate definition of the box dimension is easier to calculate and gives more accurate results than the box dimension using square mesh. It would be interesting to investigate the box dimension of other two dimensional fractals using this proposed definition. Further, generalization of this idea to $\R^n$ is still an open problem.

\bibliographystyle{amsplain}

\end{document}